\documentclass[12pt]{amsart} 
\usepackage{amsmath, amssymb, mathscinet}
\usepackage{amsthm} 
\usepackage[all]{xy}
\usepackage{enumerate}
\usepackage[T1]{fontenc}
\usepackage{hyperref} 
\usepackage{mathrsfs} 
\usepackage{graphicx}
\usepackage{cleveref}
\usepackage{color}
\renewcommand{\thefootnote}{} 
\theoremstyle{plain} 
\newtheorem{theorem}{\indent\sc Theorem}[section]
\newtheorem{lemma}[theorem]{\indent\sc Lemma}
\newtheorem{corollary}[theorem]{\indent\sc Corollary}
\newtheorem{proposition}[theorem]{\indent\sc Proposition}

\theoremstyle{definition} 
\newtheorem{definition}[theorem]{\indent\sc Definition}
\newtheorem{remark}[theorem]{\indent\sc Remark}
\newtheorem{example}[theorem]{\indent\sc Example}

\crefname{theorem}{Theorem}{Theorems}
\crefname{lemma}{Lemma}{Lemmae}
\crefname{corollary}{Corollary}{Corollaries}
\crefname{proposition}{Proposition}{Propositions}
\crefname{example}{Example}{Examples}
\crefname{section}{Section}{Sections}
\newcommand{\R}{\mathbb{R}}

\newcommand{\Z}{\mathbb{Z}}
\newcommand{\N}{\mathbb{N}}

\def\dim{\mathop{\mathrm{dim}}\nolimits}

\newcommand{\abs}[1]{\left\lvert#1\right\rvert}

\newcommand{\ds}[1]{d(#1)}

\newcommand{\Hds}[1]{\mathcal{H}(#1)}
\newcommand{\sigmacan}{\sigma}
\newcommand{\eval}{\mathrm{Ev}}
\newcommand{\basegeod}{\omega}
\newcommand{\basegeodC}{c}
\newcommand{\rays}{\mathbb{X}_\basegeod}
\newcommand{\raysR}{\rays(\R)}
\newcommand{\raysU}{X_\basegeod}
\newcommand{\raysUC}{X_\basegeodC}
\newcommand{\raysb}{\rays^b}
\newcommand{\raysbEv}[1]{\raysb(#1)}

\newcommand{\minset}{\mathrm{Min}}
\newcommand{\tsl}{s_0}

\newcommand{\showtext}[1]{
  \ifthenelse{\value{shwtxt} > 0}{#1}{}
}

\makeatletter
\def\address#1#2{\begingroup
\noindent\parbox[t]{7.8cm}{%
\small{\scshape\ignorespaces#1}\par\vskip1ex
\noindent\small{\itshape E-mail address}%
\/: #2\par\vskip4ex}\hfill%
\endgroup}%
\makeatother
\title[A topological decomposition of Busemann space]
{\uppercase{A topological product decomposition of Busemann space}}
\author{
\textsc{Tomohiro Fukaya} 
}
\date{} 
\begin{document}



\thanks{T.Fukaya was supported by JSPS KAKENHI Grant Number JP19K03471}
\renewcommand{\thefootnote}{\fnsymbol{footnote}} 

\begin{abstract}
 We show that a Busemann space $X$ which is covered by parallel bi-infinite geodesics
 is homeomorphic to a product of another Busemann space $Y$ and the real line.
 We also show that a semi-simple isometry on $X$ preserving the foliation by parallel geodesics 
 canonically induces a semi-simple isometry on $Y$.
\end{abstract}

\maketitle

\section{Introduction}
There are several generalization of the class of simply connected Riemannian 
manifolds with nonpositive sectional curvature. One of the most studied class
is that of CAT(0) spaces.
There are many results on the structure of CAT(0) spaces. One of the most fundamental one
is the following. We say that two geodesics $\gamma,\eta\colon \R\to X$ on a metric space $X$
is \textit{parallel} if $d(\gamma(t),\eta(t))$ is bounded.
\begin{theorem}[A Product Decomposition Theorem {\cite[II 2.14 ]{MR1744486}}]
\label{thm:CAT0Decomp}
 Let $X$ be a CAT(0) space and let $\basegeodC\colon\R\to X$ be a geodesic.
 \begin{enumerate}[$(1)$]
  \item \label{item:XetaCATconvex}
        Let $\raysUC$ be the union of the images of all geodesics $\gamma\colon\R\to X$
        parallel to $\basegeodC$. Then $\raysUC$ is a convex subspace of $X$.
  \item \label{item:CATdecomposition} Let $p$ be the restriction to
        $\raysUC$ of the nearest point projection from $X$ to the image
        of the geodesic $\basegeodC(\R)$. Let
        $\raysUC^{0}:=p^{-1}(\basegeodC(0))$.  Then, $\raysUC^{0}$ is
        convex, especially it is a CAT(0) space.
  \item $\raysUC$ is isometric to the product $\raysUC^{0}\times \R$.
 \end{enumerate} 
\end{theorem}

The proof of the above decomposition theorem depends on the nice
behaviour of the nearest point projection onto a bi-infinite geodesic. 
It seems hard to generalize \cref{thm:CAT0Decomp} to
more general class of metric spaces called \textit{Busemann spaces},
due to the luck of such nice property of the projections.
So we considered slightly weaker statement,
and we obtained the following topological decomposition theorem.

\begin{theorem}[A Topological Product Decomposition Theorem]
\label{thm:main}
 Let $X$ be a Busemann space and let $\basegeod\colon\R\to X$ be a geodesic.
 \begin{enumerate}[$(1)$]
  \item \label{item:raysUconvex}
        Let $\raysU$ be the union of the images of all geodesics $\gamma\colon\R\to X$
        parallel to $\basegeod$. Then $\raysU$ is a convex subspace of $X$.
  \item \label{item:XetaBusemann}
        Let $\raysR$ be the set of the images of all geodesics $\gamma\colon\R\to X$
        parallel to $\basegeod$. We equip $\raysR$ with the Hausdorff metric. Then 
        $\raysR$ is a Busemann space.
  \item \label{item:decomposition}
        $\raysU$ is homeomorphic to the product $\raysR\times \R$.
 \end{enumerate}
\end{theorem}

In the above theorem, (\ref{item:raysUconvex}) is a direct consequence of 
the \textit{Strip Lemma} (\cref{lem:strip}) by Bowditch \cite{Bow-Minkowskian}.
Our essential contribution is to show (\ref{item:XetaBusemann}) and (\ref{item:decomposition}) 
in \cref{thm:main}.



We also study the classification of isometries on $\raysR$ induced by
those of on $\raysU$. 
For a map $\basegeod\colon \R\to X$, we denote by $\bar{\basegeod}$ the map defined as
$\bar{\basegeod}(t):=\basegeod(-t)$. We say that an isometry 
$g\colon \raysU\to \raysU$ preserve the foliation by $\basegeod$ if
$g\circ \basegeod$ or $g\circ \bar{\basegeod}$ is parallel to $\basegeod$. 

\begin{theorem}
\label{thm:classfyIsometry}
 Let $X$ be a Busemann space and let $\basegeod\colon\R\to X$ be a geodesic.
 Let $g\colon \raysU\to \raysU$ be an isometry  preserving the foliation
 by $\basegeod$.
 Then $g$ induces an isometry $g_\basegeod\colon \raysR\to \raysR$ by 
 $g_\basegeod(L)=g(L)$ for $L\in \raysR$. Moreover, the following holds.
 \begin{enumerate}[$(1)$]
  \item If $g$ is elliptic, then so is $g_\basegeod$.
  \item If $g_\basegeod$ is elliptic, then $g$ is either elliptic or hyperbolic.
  \item Suppose $g$ is hyperbolic. 
        Let $\xi\colon \R\to X$ be an axis of $g$. 
        If $\xi$ is parallel to either $\basegeod$ or $\bar{\basegeod}$, 
        then $g_\basegeod$ is elliptic. Otherwise, $g_\basegeod$ is hyperbolic.
  \item If $g_\basegeod$ is hyperbolic, then so is $g$.        
 \end{enumerate}
 Especially, $g$ is semi-simple if and only if so is $g_\basegeod$.
\end{theorem}

The organization of the paper is as follows. 
In \cref{sec:comb-busem-space}, 
we briefly review the definition of Busemann spaces
and some terminologies on geodesic bicombings. We will remark a criterion that provide
a sufficient condition for a geodesic space to be a Busemann space.
In \cref{sec:parall-geod-busem}, 
we construct a geodesic bicombing on the space of parallel geodesics and
we prove (\ref{item:XetaBusemann}) of 
\cref{thm:main}.
In \cref{sec:busem-funct-topol}, 
we introduce the Busemann functions and 
we prove (\ref{item:decomposition}) of \cref{thm:main}.
In \cref{sec:isometry}, we review the standard notations on
the classification of isometries, and then
we prove \cref{thm:classfyIsometry}.

\section{Combing and Busemann space}
\label{sec:comb-busem-space}
\subsection{Busemann space}
Let $(X,d)$ be a metric space. A map $\gamma\colon I \to X$
where $I\subset \R$ is a closed interval
is a \textit{geodesic}
if for all $t,s\in I$, we have 
$\ds{\gamma(s),\gamma(t)}=\abs{t-s}$.
We say that $(X,d)$ is a \textit{geodesic space} if for any pair $(x,y)\in X\times X$, 
there exists a geodesic $\gamma\colon [0,\ds{x,y}]\to X$ with $\gamma(0)=x$ and 
$\gamma(\ds{x,y})=y$.
We also say that $(X,d)$ is a \textit{uniquely geodesic space} 
if it is a geodesic space and 
for any pair $(x,y)\in X\times X$, a geodesic between $x$ and $y$ is unique.
A \textit{linearly reparametrised geodesic} is a map of the form $[t\mapsto \gamma(\lambda t)]$
where $\gamma$ is a geodesic and $\lambda$ is a positive real number.

\begin{definition}
 A geodesic space $(X,d)$ is a \textit{Busemann space} if, for any two 
 linearly reparametrised geodesics $\gamma,\eta\colon [0,1]\to X$, 
 the map $[t\mapsto \ds{\gamma(t),\eta(t)}]$ is convex in $t$.
\end{definition}

It is easy to see that a Busemann space is a uniquely geodesic space.

\subsection{Geodesic bicombing}
A geodesic bicombing $\sigma$ on $X$ is a map
\begin{align*}
 \sigma\colon X\times X \times [0,1] \to X
\end{align*}
such that for every pair $(x,y)\in X\times X$, the map 
$\sigma_{xy}:=\sigma(x,y,\cdot)\colon [0,1]\to X$ is a linearly reparametrised geodesic
with $\sigma_{xy}(0)=x$ and $\sigma_{xy}(1)=y$. 
\begin{enumerate}[(1)]
 \item We say that $\sigma$ is \textit{conical} if for all $x,y,x',y'\in X$ and $t\in [0,1]$, 
       we have
       \begin{align*}
        \ds{\sigma_{xy}(t),\sigma_{x'y'}(t)} \leq (1-t)\ds{\sigma_{xy}(0),\sigma_{x'y'}(0)} 
        + t\ds{\sigma_{xy}(1),\sigma_{x'y'}(1)}.
       \end{align*}
 \item We say that $\sigma$ is \textit{convex} if the map 
$[t\mapsto \ds{\sigma_{xy}(t),\sigma_{x'y'}(t)}]$ is convex for all $x,y,x',y'\in X$, that is,
 for any $a,b\in [0,1]$ and $t\in [0,1]$, we have
\begin{align*}
 &\ds{\sigma_{xy}((1-t)a+tb),\sigma_{x'y'}((1-t)a+tb)} \\
 &\leq (1-t)\ds{\sigma_{xy}(a),\sigma_{x'y'}(a)} + t\ds{\sigma_{xy}(b),\sigma_{x'y'}(b)}.
\end{align*}
 \item We say that $\sigma$ is \textit{consistent} if for all $x,y\in X$, 
       $0\leq a\leq b\leq 1$, $p:=\sigma_{xy}(a)$, $q:=\sigma_{xy}(b)$, 
       and $t\in [0,1]$, we have
       \begin{align*}
        \sigma_{pq}(t) = \sigma_{xy}((1-t)a+ tb).
       \end{align*}
\end{enumerate}
\begin{example}
\label{ex:caninical-bicombing} 
 Let $(X,d)$ be a uniquely geodesic space.
 Then $X$ admits a 
 geodesic bicombing $\sigmacan$ such that 
 $\sigmacan_{xy}$ is a unique linearly reparametrised
 geodesic between $x$ and $y$. 
 We call this the \textit{canonical bicombing}.
 Obviously, the canonical bicombing is consistent.
\end{example}

We remark that if a geodesic bicombing is conical and consistent, then it is convex.
Thus we have the following criterion.
\begin{lemma}
\label{lem:criterion_being_Busemann}
 Let $X$ be a uniquely geodesic space. 
 If the canonical bicombing is conical, then it is convex.
 Especially, $X$ is a Busemann space.
\end{lemma}






\section{Parallel geodesics in Busemann space}
\label{sec:parall-geod-busem}
Let $(X,d)$ be a Busemann space. 

\begin{proposition}
 Let $\gamma\colon \R\to X$ be a geodesic. For $x\in X$, there exists unique
 $y\in \gamma(\R)$ such that 
$\ds{x,y} = \ds{x,\gamma(\R)} = \inf\{\ds{x,w}:w\in \gamma(\R)\}$.
\end{proposition}
The following argument is due to Kato \cite{kato-thesis}.
\begin{proof}
 Since $\R$ is proper, so is the image $\gamma(\R)$. It follows that
 there exists $y\in \gamma(\R)$ such that $\ds{x,y} = \ds{x,\gamma(\R)}$.
 
 Now we show that such $y$ is unique. Set $r:=\ds{x,\gamma(\R)}$.
 Suppose that there exists $y,z\in \gamma(\R)$
 with $y\neq z$ such that $\ds{x,y} = \ds{x,z} = r$. 
 Let $\sigmacan$ be the canonical bicombing of $X$. 
 Set $m = \sigmacan_{yz}(1/2)$. By the convexity, and the minimality, $\ds{x,m} = r$.
 By the convexity of the map
 $[t\mapsto \ds{\sigmacan_{yx}(t),\sigmacan_{yz}(t)}]$, we have
 $\ds{\sigmacan_{yx}(1/2),m}\leq r/2$. By the triangle inequality,
 \begin{align*}
  r = \ds{x,m} \leq \ds{x,\sigmacan_{yx}(1/2)}+\ds{\sigmacan_{yx}(1/2),m}\leq r.
 \end{align*}
 So $\ds{\sigmacan_{yx}(1/2),m} = r/2$. 
 By the same argument, we have $\ds{\sigmacan_{zx}(1/2),m} = r/2$. 
 Then by the uniqueness of the geodesic between $x$ and $m$, we have 
 $\sigmacan_{yx}(1/2)=\sigmacan_{zx}(1/2)$. 
 So set $x(1):=\sigmacan_{yx}(1/2)=\sigmacan_{zx}(1/2)$. 
 Replacing $x$ by $x(1)$, we apply the same argument and 
 we obtain $x(2)\in X$ such that 
 $\ds{x(2),y} = \ds{x(2),m} = \ds{x(2),z} = r/2^{2}$.
 Now we can apply inductively
 the same arguments and obtain $x(n)\in X$ for $n\in \N$ such that 
 $\ds{x(n),y} = \ds{x(n),m} = \ds{x(n),z} = r/2^{n}$. It follows that
 $\ds{y,z} \leq r/2^{n-1}\to 0 $ as $n\to \infty$. This is a contradiction.
\end{proof}

\begin{definition}
 We say that two bi-infinite geodesics
 $\gamma,\eta\colon \R\to X$ are \textit{parallel} if $\ds{\gamma(t),\eta(t)}$ is bounded.
 This define an equivalence relation on the set of bi-infinite geodesics.
\end{definition}

\begin{lemma}
 \label{lem:parallel-const}
 Let $\gamma,\eta$ be parallel geodesics. For any constant $C\in \R$, the function
 \begin{align*}
  \R\ni t\mapsto \ds{\gamma(t),\eta(t+C)}\in \R
 \end{align*}
 is constant.
\end{lemma}

\begin{proof}
 Since the map $[t\mapsto \ds{\gamma(t),\eta(t+C)}]$ is convex, if it is bounded,
 it must be constant.
\end{proof}

For closed subsets $H,K\subset X$, we denote by $\Hds{H,K}$ the Hausdorff metric between
$H$ and $K$. For bi-infinite geodesics $\gamma,\eta\colon \R\to X$, we abbreviate
the Hausdorff metric of their images $\Hds{\gamma(\R),\eta(\R)}$ to
$\Hds{\gamma,\eta}$.

\begin{corollary}
\label{cor:Parallel-Haus2}
 Let $\gamma,\eta$ be parallel geodesics.
 For any $t\in \R$, we have
 \begin{align*}
   \Hds{\gamma,\eta} = d(\gamma(t),\eta(\R)) =  \inf_{s\in \R}\ds{\gamma(t),\eta(s)}.
 \end{align*}
\end{corollary}

Bowditch showed the following Strip Lemma.
\begin{lemma}[Strip Lemma \cite{Bow-Minkowskian}]
\label{lem:strip}
 Let $(X,d)$ be a Busemann space. Let $r>0$. We suppose that 
 $\gamma,\gamma'\colon \R\to X$ are parallel geodesics such that 
 $\ds{\gamma(t),\gamma'(t)}=r$.
 For $t\in \R$, let $\beta_t\colon [0,r]\to X$ be the geodesic from 
 $\gamma(t)$ to $\gamma'(t)$. Define a map $\beta\colon \R\times [0,r]\to X$ 
 by $\beta(t,u):=\beta_t(u)$. Then:
 \begin{enumerate}[(1)]
  \item for each $u\in [0,r]$, the map $\gamma_u:=[t\mapsto \beta(t,u)]$
        is a bi-infinite geodesic parallel to $\gamma$ and $\gamma'$;
  \item for each $t_0\in \R$ and $\lambda\in \R$, the maps
        $[u\mapsto \beta(t_0 + \lambda u,u)]$ is a linearly reparametrised geodesic.
 \end{enumerate}
\end{lemma}


\begin{lemma}
\label{lem:flatstrip-geod}
 Let $\gamma,\gamma'\colon \R\to X$ be parallel geodesics
 such that $\ds{\gamma(0),\gamma'(0)} = \Hds{\gamma,\gamma'}$.
 Set $r= \Hds{\gamma,\gamma'}$.
 For $u\in [0,r]$, let $\gamma_u$ be the geodesic given by \cref{lem:strip}
 which is parallel to $\gamma$ and $\gamma'$. Then for $u,v\in [0,r]$, we have
 \begin{align*}
   \Hds{\gamma_u,\gamma_v} = \abs{u-v}.
 \end{align*}
\end{lemma}

\begin{proof}
 Let $\beta_0\colon [0,r]\to X$ be the geodesic from $\gamma(0)$ to $\gamma'(0)$.
 For $u\in [0,r]$, by the construction of $\gamma_u$, we have $\gamma_u(0) = \beta_0(u)$, so
 $\beta_0(u)\in \gamma_u(\R)$.
 Then for $0\leq u\leq v\leq r$, we have 
 \begin{align*}
  \Hds{\gamma_u,\gamma_v} \leq \ds{\beta_0(u),\beta_0(v)} = \abs{u-v}.
 \end{align*}
 Suppose $\Hds{\gamma_u,\gamma_v}< \abs{u-v}$. Then 
 \begin{align*}
  r &= \Hds{\gamma,\gamma'} \leq  \Hds{\gamma,\gamma_u} + \Hds{\gamma_u,\gamma_v}
  + \Hds{\gamma_v,\gamma'}\\
  &<  \ds{\beta_0(0),\beta_0(u)} + \ds{\beta_0(u),\beta_0(v)} 
  + \ds{\beta_0(v),\beta_0(r)} =  r.
 \end{align*}
 This is contradiction. Therefore we have $\Hds{\gamma_u,\gamma_v} = \abs{u-v}$.
\end{proof}

The following is the key to the uniqueness of geodesics in $\raysR$.

\begin{lemma}
\label{lem:3geods-line}
 Let $\gamma_1,\gamma_2,\gamma_3$ be a pairwise parallel bi-infinite geodesics in $X$
 such that
 \begin{align*}
  \Hds{\gamma_1,\gamma_3}=\Hds{\gamma_1,\gamma_2} 
  + \Hds{\gamma_2,\gamma_3}.
 \end{align*}
 Let $x\in \gamma_1(\R)$, $y\in \gamma_2(\R)$ and $z\in \gamma_3(\R)$. 
 If $\ds{x,y}=\Hds{\gamma_1,\gamma_2}$ and $\ds{x,z}=\Hds{\gamma_1,\gamma_3}$ holds,
 then we have $\ds{y,z}=\Hds{\gamma_2,\gamma_3}$.
\end{lemma}

\begin{proof}
 We choose $z'\in \gamma_3(\R)$ such that $\ds{y,z'}=\Hds{\gamma_2,\gamma_3}$. Then we have
 \begin{align*}
  \Hds{\gamma_1,\gamma_3}&\leq \ds{x,z'}\leq \ds{x,y}+\ds{y,z'}\\
  &=\Hds{\gamma_1,\gamma_2} + \Hds{\gamma_2,\gamma_3} = \Hds{\gamma_1,\gamma_3}.
 \end{align*}
 It follows that $\ds{x,z'} = \Hds{\gamma_1,\gamma_3}$. 
 Since such $z'\in \gamma_3(\R)$ is unique, we have $z=z'$, and so 
 $\ds{y,z}=\Hds{\gamma_2,\gamma_3}$.
\end{proof}



\begin{definition}
 Let $\gamma\colon \R\to X$ be a geodesic. We denote by $\rays$ the set of 
 all geodesic $\gamma\colon \R\to X$ which is parallel to $\basegeod$. 
 We also denote by
 $\raysR$ the set of images of geodesic in $\rays$,
 and denote by $\raysU$ the union of the images of geodesic in $\rays$. Thus we have
 \begin{align*}
  \raysR = \{\gamma(\R):\gamma\in \rays\}, \quad
  \raysU = \bigcup_{\gamma\in \rays}\gamma(\R) = \bigcup_{L\in \raysR}L.
 \end{align*}
\end{definition}

It follows from \cref{lem:strip} that $\raysU$ is a convex subspace of $X$.
We equip $\raysR$ with a Hausdorff metric $\mathcal{H}$ 
so that $(\raysR,\mathcal{H})$ is a metric space.
Now we construct a geodesic bicombing on $\raysR$.

\begin{definition}
 Let $\gamma,\eta \in \rays$. We suppose that $r:=\ds{\gamma(0),\eta(0)}=\Hds{\gamma,\eta}$.
 Let $\beta\colon \R \times [0,r]\to X$ be a map given in \cref{lem:strip}. 
 For $u\in [0,r]$, set $\gamma_u = [t \mapsto \beta(t,u)]$ as in \cref{lem:strip}. 
 Then a map $[u\mapsto \gamma_u(\R)]$ is a geodesic from $\gamma(\R)$ to $\eta(\R)$
 by \cref{lem:flatstrip-geod}.
 Therefore we define 
 $\Sigma(\gamma,\eta,u)=\Sigma_{\gamma\eta}(u)=[u\mapsto \gamma_{ru}(\R)]$ 
 for $u\in [0,1]$. 
\end{definition}

\begin{remark}
 Let $\gamma,\eta \in \rays$. We suppose that $r:=\ds{\gamma(0),\eta(0)}=\Hds{\gamma,\eta}$.
 Let $a\in \R$. We define $\gamma^a,\eta^a\colon \R\to X$ by $\gamma^a(t):=\gamma(t+a)$ and 
 $\eta^a(t):=\eta(t+a)$ for $t\in \R$. Then we have 
 $\Sigma_{\gamma\eta} = \Sigma_{\gamma^a \eta^a}$.
\end{remark}

It follows that $\raysR$ is a geodesic space. In fact, we have the following.

\begin{proposition}
\label{prop:raysRuniqueGoed}
 $\raysR$ is a uniquely geodesic space. 
 Therefore $\Sigma$ is the canonical bicombing.
\end{proposition}

\begin{proof}
 Let $\gamma,\eta \in \rays$. We suppose that $r:=\ds{\gamma(0),\eta(0)}=\Hds{\gamma,\eta}$.
 Let $\Gamma\colon [0,r]\to \raysR$ be a geodesic
 such that $\Gamma(0)=\gamma(\R)$ and $\Gamma(r)=\eta(\R)$. 
 Now we define a map $\xi\colon[0,r]\to X$ as follows.
 Set $\xi(0):=\gamma(0)\in\Gamma(0)$. For $u\in [0,r]$, we choose
 unique $\xi(u)\in \Gamma(u)$ such that 
 $\ds{\xi(0),\xi(u)} = \Hds{\Gamma(0),\Gamma(u)}$. 
 Then for $s,t\in [0,r]$, by
 \cref{lem:3geods-line}, we have 
 \begin{align*}
  \ds{\xi(s),\xi(t)}=\Hds{\Gamma(s),\Gamma(t)}=\abs{s-t}.
 \end{align*}
 Especially, $\ds{\gamma(0),\xi(r)}=r=\ds{\gamma(0),\eta(0)}$ and 
 $\xi(r)\in \Gamma(r)=\eta(\R)$. So $\xi(r)=\eta(0)$.
 Therefore $\xi$ is a geodesic between
 $\gamma(0)$ and $\eta(0)$. Since $X$ is a uniquely geodesic space, 
 $\xi$ is equal to the geodesic $[u\mapsto \beta(0,u)]$
 where $\beta$ is a map given by \cref{lem:strip} for $\gamma,\eta$. It follows that 
 $\Gamma$ is equal to the geodesic $\Sigma_{\gamma\eta}$.
\end{proof}

\begin{lemma}
 Let $\sigmacan$ be the canonical bicombing of $X$, and let $\Sigma$ be
 that of $\raysR$. For $\gamma_1,\gamma_2\in \rays$, and 
 $x_1\in \gamma_1(\R),\,x_2\in \gamma_2(\R)$, we have
 \begin{align*}
  \sigmacan_{x_1 x_2}(t) \in \Sigma_{\gamma_1 \gamma_2}(t) \quad (t\in [0,1]).
 \end{align*}
\end{lemma}

\begin{proof}
 We can assume that $x_1=\gamma_1(0)$ and
 $r:= \ds{\gamma_1(0),\gamma_2(0)} = \Hds{\gamma_1,\gamma_2}$.  Then
 $x_2=\gamma_2(\lambda)$ for some $\lambda\in \R$.  Let $\beta\colon
 \R\times [0,r]\to X$ be a map given by \cref{lem:strip} for parallel
 geodesics $\gamma_1,\gamma_2$.  Then the map $[u\mapsto
 \beta(\lambda u,ru);u\in [0,1]]$ is a unique linearly reparametrised
 geodesics between $x_1$ and $x_2$, so it coincides with $\sigmacan_{x_1 x_2}$.
 It follows that $\sigmacan_{x_1 x_2}(t) = \beta(\lambda t,rt)\in \Sigma_{\gamma_1 \gamma_2}(t)$.
\end{proof}

\begin{theorem}
 $\raysR$ is a Busemann space.
\end{theorem}

\begin{proof}
 We denote by $\Sigma$ the canonical bicombing of $\raysR$.  By
 \cref{prop:raysRuniqueGoed}, $\raysR$ is a uniquely geodesic space.  So by
 \cref{lem:criterion_being_Busemann}, it is enough to show that the
 canonical bicombing $\Sigma$ is conical.

 Let $\gamma_1,\gamma_2,\eta_1,\eta_2\in \rays$. 
 We choose $x_i\in \gamma_i(\R),\,y_i\in \eta_i(\R)$ for $i=1,2$ such that
 \begin{align*}
  \ds{x_1,y_1} = \Hds{\gamma_1,\eta_1}, \quad
  \ds{x_2,y_2} = \Hds{\gamma_2,\eta_2}.
 \end{align*}
 Let $\sigmacan$ be the canonical bicombing of $X$. For $t\in [0,1]$, we have
 \begin{align*}
  \sigmacan_{x_1 x_2}(t)\in \Sigma_{\gamma_1 \gamma_2}(t), \quad
  \sigmacan_{y_1 y_2}(t)\in \Sigma_{\eta_1 \eta_2}(t).
 \end{align*}
 Thus we have
 \begin{align*}
  \Hds{\Sigma_{\gamma_1 \gamma_2}(t),\Sigma_{\eta_1 \eta_2}(t)}
  &\leq  \ds{\sigmacan_{x_1 x_2}(t),\sigmacan_{y_1 y_2}(t)}\\
  &\leq (1-t)\ds{x_1,y_1} + t\ds{x_2,y_2} \\
  &=  (1-t)\Hds{\gamma_1,\eta_1} + t\Hds{\gamma_2,\eta_2}.
 \end{align*}
 Therefore $\Sigma$ is conical and this completes the proof.
\end{proof}

\section{Busemann function and topological splitting}
\label{sec:busem-funct-topol}
First we introduce the Busemann function.
\begin{definition}
 Let $X$ be a metric space
 and $\eta\colon \R\to X$ be an geodesic.
 We define a function 
 $b_\eta\colon X\to \R$ by
 \begin{align*}
  b_\eta(x)= \lim_{t\to \infty}(\ds{x,\eta(t)}-t),
 \end{align*}
 and call it the Busemann function associated to $\eta$.
\end{definition}

\begin{remark}
 Since $\ds{x,\eta(t)}-t$ is non-increasing on $t$ and bounded below, the limit exists,
 and $b_\eta$ is a 1-Lipschitz function.
\end{remark}

\begin{lemma}
\label{lem:Busemann-convex}
 Let $X$ be a Busemann space and
 $\eta\colon \R\to X$ be a geodesic.
 For a geodesic $\xi\colon \R\to X$, the composite
 $b_\eta\circ \xi\colon \R\to \R$ is convex. Moreover, 
 if 
 $\xi$ and $\eta$ are parallel,
 it is strictly decreasing.
\end{lemma}

\begin{proof}
 Since the canonical bicombing is convex, the function 
 $s\mapsto \ds{x,\xi(t)}$ is convex,
 and so are its translates. Since the Busemann function $b_\eta$ 
 is the pointwise limit of the function 
 $[s\mapsto \ds{\xi(s),\eta(t)}-t]$, it is convex. 

 Suppose $\sup_{t\in \R}\ds{\eta(t),\xi(t)}<\infty$. Then 
 \begin{align*}
  \lim_{s\to \infty}b_\eta(\xi(s))=-\infty,\quad
  \lim_{s\to -\infty}b_\eta(\xi(s))=\infty,
 \end{align*}
 it follows that $b_\eta\circ\xi$ is strictly 
 decreasing.
\end{proof}


Let $X$ be a Busemann space and
$\basegeod\colon \R\to X$ be a geodesic. 
We define a subset $\raysb$ of $\rays$ and
a subset $\raysbEv{t}$ of $X$ where $t\in \R$ as follows:
\begin{align*}
 \raysb:=& 
 \left\{\gamma\in \rays: b_\basegeod(\gamma(0))=0\right\},\\
 \raysbEv{t}:=& \{\gamma(t):\gamma\in \raysb\}.
\end{align*}
There exists a bijection between $\raysb$ and $\raysR$. Therefore we
identify $\raysb$ with $\raysR$, and we regard $(\raysb,\mathcal{H})$ as a metric
space.

\begin{lemma}
\label{lem:proj_conti}
 For $t\in \R$, define a map $\eval_t\colon \raysb\to \raysbEv{t}$ by 
 $\eval_t(\gamma):=\gamma(t)$. Then $\eval_t$ is continuous.
\end{lemma}
\begin{proof}
 It is enough to show that $\eval_0$ is continuous. Indeed, for
 $\gamma,\gamma'\in \raysb$, by \cref{lem:parallel-const}, we have
 $\ds{\gamma(t),\gamma'(t)} = \ds{\gamma(0),\gamma'(0)} $. Therefore, if
 $\eval_0$ is continuous, so is $\eval_t$ for all $t\in \R$.

 Let $\gamma\in \raysb$. We fix $\epsilon>0$. Set 
 $\delta':=\min\{\abs{b_\basegeod(\gamma(\epsilon/2))},\, b_\basegeod(\gamma(-\epsilon/2))\}$. By Lemma~\ref{lem:Busemann-convex}, we have
 $\delta'>0$. Set $\delta:=\min\{\delta'/2, \epsilon/2\}$. 
 Here we use Lemma~\ref{lem:Busemann-convex} again, we have
 \begin{align}
  \label{eq:1}
  \abs{b_\basegeod(\gamma(t))} > \delta' \quad \text{for all } t\in \R \text{ with } \abs{t}>\epsilon/2.
 \end{align}

 Now let $\gamma'\in \raysb$ such that
 $\Hds{\gamma,\gamma'}<\delta$. There exists $t\in \R$ such that 
 $\ds{\gamma(t),\gamma'(0)}<\delta$.
 We suppose that $\ds{\gamma(0),\gamma'(0)}>\epsilon$.
 Then we have
 \begin{align*}
  \abs{t} = \ds{\gamma(0),\gamma(t)} \geq \ds{\gamma(0),\gamma'(0)} - \ds{\gamma'(0),\gamma(t)}> \epsilon - \delta\geq \epsilon/2.
 \end{align*}
 Since $b_\basegeod$ is 1-Lipschitz, using (\ref{eq:1}), we have 
 \begin{align*}
  \abs{b_\basegeod(\gamma'(0))}>\abs{b_\basegeod(\gamma(t))} -\delta\geq \delta'/2>0.
 \end{align*}
 This contradicts that 
 $\gamma'\in \raysb$. Thus we have $\ds{\gamma(0),\gamma'(0)}\leq \epsilon$. 
 Therefore $\eval_0$ is continuous.
\end{proof}


\begin{theorem}
\label{thm:topdecomp}
 The map
 \begin{align*}
  \eval\colon \raysb \times \R \to \raysU,
  \quad \eval(\gamma,t) = \gamma(t)
 \end{align*}
 is homeomorphism.
\end{theorem}
\begin{proof}
 First we show that $\eval$ is continuous. By \cref{lem:proj_conti}, 
 for $(\gamma,t)\in \raysb \times \R$ and $\epsilon>0$, 
 there exists $\delta>0$ such that for all $\gamma'\in \raysb$ 
 with $\Hds{\gamma,\gamma'}<\delta$, we have 
 $\ds{\eval_t(\gamma),\eval_t(\gamma')}=\ds{\gamma(t),\gamma'(t)}<\epsilon/2$. 
 Then for $(\gamma',t')\in \raysb\times \R$ 
 with $\Hds{\gamma,\gamma'}<\delta$ and $\abs{t-t'}< \epsilon/2$, we have
 \begin{align*}
  \ds{\eval(\gamma,t),\eval(\gamma',t')} = \ds{\gamma(t),\gamma'(t')} \leq 
  \ds{\gamma(t),\gamma'(t)} + \ds{\gamma'(t),\gamma'(t')} <\epsilon.
 \end{align*}
 Therefore $\eval$ is continuous. 

 Now we will see that the inverse is also continuous. 
 We define a map $q\colon \raysU\to \raysb$ such that
 for $x\in \raysU$, we have $x\in q(x)(\R)$. 
 Since for $x\in \raysU$ there exists unique $\gamma\in \raysb$ with $x\in \gamma(\R)$,
 this map is well-defined. 
 For $x,x'\in \raysR$, 
 by \cref{cor:Parallel-Haus2}, we have $\Hds{q(x),q(x')}\leq \ds{x,x'}$. 
 So the map $q$ is continuous. We also define a map $\theta\colon \raysU\to \{1,-1\}$ by
 \begin{align*}
  \theta(x) = \begin{cases}
	       1 & \text{if } x\in q(x)([0,\infty))\\
	       -1 & \text{else}
	      \end{cases}
 \end{align*}
 Then the map $x\mapsto \theta(x)\ds{x,q(x)(0)}$ is continuous. 
 Since the inverse $\eval^{-1}$ is given by 
 $\eval^{-1}(x)=(q(x),\theta(x)\ds{x,q(x)(0)})$, it is continuous.
\end{proof}

\begin{corollary}
 The map $\eval_t$ defined in \cref{lem:proj_conti}
 is a homeomorphism.
\end{corollary}

Using the fact from the dimension theory \cite[Theorem 5.11.]{Morita-Cech-cohom}, 
we have the following.
\begin{corollary}
\label{cor:dim}
 If the topological dimension of $X$ is finite, we have
 \begin{align*}
  \dim\raysR \leq  \dim X -1.
 \end{align*}
\end{corollary}

\begin{lemma}
\label{lem:BfuncTransl}
 Let $\gamma$ be a geodesic which is parallel to $\basegeod$. 
 For any $u,s\in \R$,
 \begin{align*}
  b_\basegeod(\gamma(u+s)) = b_\basegeod(\gamma(u))-s.
 \end{align*}
\end{lemma}
\begin{proof}
 By \cref{lem:parallel-const},
 \begin{align*}
   b_\basegeod(\gamma(u+s)) 
   &=\lim_{t\to \infty}\left\{\ds{\gamma(u+s),\basegeod(t)}-t\right\}\\
   &=\lim_{t\to \infty}\left\{\ds{\gamma(u),\basegeod(t-s)}-(t-s) -s\right\}\\
   &=b_\basegeod(\gamma(u)) -s.
  \end{align*}
\end{proof}

\begin{corollary}
\label{cor:levelset}
 For $\gamma\in \raysb$ and $t\in \R$, we have
 $\eval(\gamma,t) \in b_\basegeod^{-1}(-t)$.
\end{corollary}


\section{isometries}
\label{sec:isometry} In this section, we consider a classification of
isometries induced on the first factor of the decomposition in \cref{thm:main}.

First we review standard terminologies. 
Let $X$ be a metric space and
$f\colon X\to X$ be an isometry. The \textit{displacement function} 
$d_f\colon X\to \R$ of $f$ is defined as $d_f(x)=\ds{x,f(x)}$ for $x\in X$. 
The \textit{minimal displacement} of $f$, denoted by $\abs{f}$, 
is the infimum of $d_f$, that is, 
\begin{align*}
 \abs{f}=\inf_{x\in X} d_f(x).
\end{align*}
The \textit{minimal set} of $f$, denoted by $\minset(f)$, is the subset
defined as
\begin{align*}
 \minset(f) = \{x\in X: d_f(x)=\abs{f}\}.
\end{align*}
We say that an isometry $f\colon X\to X$ is
\begin{enumerate}[(i)]
 \item \textit{parabolic} if $\minset(f) = \emptyset$, 
 \item \textit{elliptic} if $\minset(f)\neq \emptyset$ and $\abs{f}=0$,
 \item \textit{hyperbolic} if $\minset(f)\neq \emptyset$ and $\abs{f}>0$,
 \item \textit{semi-simple} if $f$ is elliptic or hyperbolic.
\end{enumerate}
It is easy to see that $f$ is elliptic if and only if $f$ has a fixed point.

Let $f\colon X\to X$ be an isometry. 
We say that $f$ is \textit{axial} if $f$ has no fixed point and
there exists a geodesic $\gamma\colon \R\to X$ whose image is setwise
invariant by $f$. We call such $\gamma$ an \textit{axis} of $f$.



\begin{proposition}[{\cite[Proposition 11.2.10]{MR2132506}}]
 Let $f\colon X\to X$ be an axial isometry and
 $\gamma\colon \R\to X$ be an axis of $f$.
 Set $b_\gamma(f)= - b_\gamma(f(\gamma(0)))$. 
 Then we have $\abs{b_\gamma(f)}=\abs{f}$ and 
 for all $t\in \R$, we have
 \begin{align*}
  f\circ\gamma(t) = \gamma(t + b_\gamma(f)).
 \end{align*}
\end{proposition}

Axial isometries are hyperbolic (\cite[Corollary 11.2.9]{MR2132506}). 
The converse holds if $X$ is a Busemann space.

\begin{proposition}[{\cite[Corollary 11.4.4]{MR2132506}}]
 Let $X$ be a Busemann space and $f\colon X\to X$ be an isometry. Then 
 $f$ is hyperbolic if and only if $f$ has an axis.
\end{proposition}

Let $X$ be a Busemann space and let $\basegeod\colon \R\to X$ be a geodesic.
We define a geodesic $\bar{\basegeod}\colon \R\to X$ by 
$\bar{\basegeod}(t)=\basegeod(-t)$.

\begin{definition}
 We say that an isometry 
$g\colon \raysU\to \raysU$ \textit{preserve the foliation} by $\basegeod$ if
$g\circ \basegeod$ or $g\circ \bar{\basegeod}$ is parallel to $\basegeod$. 
\end{definition}

Let $g\colon \raysU\to \raysU$ be an isometry preserving the foliation
by $\basegeod$.
For $\gamma,\eta\in \rays$, we have 
$g\circ \gamma(\R),g\circ \eta(\R)\in \raysR$, and 
$\Hds{\gamma(\R),\eta(\R)}=\Hds{g\circ\gamma(\R),g\circ\eta(\R)}$.
Thus we define an isometry $g_\basegeod\colon \raysR\to \raysR$ by
$g_\basegeod(\gamma(\R)):= g\circ \gamma(\R)$.

\begin{proposition}
\label{prop:elipA}
 Let $g\colon \raysU\to \raysU$ be an isometry preserving the foliation by $\basegeod$.
 If $g$ is elliptic, then so is $g_\basegeod$.
\end{proposition}

\begin{proof}
 If $g$ is elliptic, then there exists $x\in \raysU$ such that $g(x)=x$. Then for
 $L\in \raysR$ with $x\in L$, we have $g_\basegeod(L) = L$.
 It follows that $g_\basegeod$ is elliptic.
\end{proof}

\begin{proposition}
\label{prop:elipB} 
 Let $g\colon \raysU\to \raysU$ be an isometry
 preserving the foliation by $\basegeod$. If $g_\basegeod$ is elliptic,
 then $g$ is either elliptic or hyperbolic.
\end{proposition}

\begin{proof}
 If $g_\basegeod$ is elliptic, then $g_\basegeod$ has a fixed point, that is,
 there exists $\gamma\in \rays$ such that $g\circ \gamma(\R) = \gamma(\R)$.
 So $\gamma$ is an invariant geodesic of $g$. Therefore $g$ is elliptic or hyperbolic.
\end{proof}

\begin{proposition}
\label{prop:hypA}
 Let $g\colon \raysU\to \raysU$ be a hyperbolic isometry preserving the foliation by $\basegeod$.
 Let $\xi\colon \R\to X$ be an axis of $g$. 
 If $\xi$ is parallel to either $\basegeod$ or $\bar{\basegeod}$, 
 then $g_\basegeod$ is elliptic. Otherwise, $g_\basegeod$ is hyperbolic.
\end{proposition}

\begin{proof}
 Suppose that $\xi$ is parallel to either $\basegeod$ or $\bar{\basegeod}$, 
 Then $g\circ \basegeod(\R) = \basegeod(\R)$. So $g_\basegeod$ is elliptic.
 
 Now we suppose that $\xi$ is parallel to neither $\basegeod$ nor $\bar{\basegeod}$.
 We define a map $\Gamma\colon \R\to \raysR$ such that $\xi(t)\in \Gamma(t)$. 
 Since $\xi$ is parallel to neither $\basegeod$ or $\bar{\basegeod}$, 
 such $\Gamma(t)$ in $\raysR$ uniquely exists, so $\Gamma$ is well-defined.
 Then $\Gamma$ is a geodesic. Indeed, let $\sigma$ be the canonical bicombing of $X$, and
 let $\Sigma$ be that of $\raysR$.
 For $u,v\in \R$ with $u\leq v$ and for $t\in [u,v]$, 
 by the construction of $\Sigma$, we have 
 \begin{align}
  \xi(t) = \sigma_{\xi(u)\xi(v)}((t-u)/(v-u)) 
  \in \Sigma_{\Gamma(u) \Gamma(v)}((t-u)/(v-u)).
 \end{align}
 It follows that $\Gamma(t) = \Sigma_{\Gamma(u) \Gamma(v)}((t-u)/(v-u))$. Therefore
 $\Gamma$ is a geodesic.
 
 Since $g(\xi(\R)) = \xi(\R)$, we have $g_\basegeod(\Gamma(\R)) = \Gamma(\R)$. 
 Thus $\Gamma$ is an invariant bi-infinite geodesic 
 on which $g_\basegeod$ acts as a translation. 
 Therefore $g_\basegeod$ is hyperbolic.
\end{proof}

The following Lemma is a direct consequence of \cref{lem:strip} and the uniqueness of
geodesic in $\raysR$ (\cref{prop:raysRuniqueGoed}).
\begin{lemma}[Flat plane]
\label{lem:flat-plane}
 Let $\Gamma\colon \R\to \raysR$ be a bi-infinite geodesic in $\raysR$.
 We choose $x\in \Gamma(0)$ as a base point. For $u\in \R$, we choose $\gamma_u\in \rays$
 such that, $\gamma_0(0)=x$, and for all $u\in \R$, $\gamma_u(\R)=\Gamma(u)$ and
 \begin{align*}
  \ds{x,\gamma_u(0)} = \ds{\gamma_0(0},\gamma_u(0)) = \Hds{\gamma_0,\gamma_u}=u.
 \end{align*}
 We define a map $\beta\colon \R^2\to\raysU$ by $\beta(t,u):=\gamma_u(t)$. Then a map
 $[u\mapsto \beta(t_0 + \lambda u,u)]$ is a linearly reparametrised geodesic where
 $t_0,\lambda\in \R$ are constants.
\end{lemma}

\begin{proposition}
\label{prop:hypB} Let $g\colon \raysU\to \raysU$ be an isometry
 preserving the foliation by $\basegeod$. If $g_\basegeod$ is
 hyperbolic, so is $g$.
\end{proposition}

\begin{proof}
 Suppose $g_\basegeod$ is hyperbolic. 
 For simplicity, we assume that $g\circ \basegeod$ is parallel to $\basegeod$.

 Let $\Gamma\colon \R\to \raysR$ be an axis of $g_\basegeod$. 
 We choose a base point $x\in \Gamma(0)$. For $u\in \R$, we choose $\gamma_u\in \rays$ as 
 in \cref{lem:flat-plane} and define the map $\beta$ by $\beta(t,u):=\gamma_u(t)$.
 Let $r>0$ be the translation length of $g_\basegeod$. 
 Then we have $g(\gamma_0(\R))=\gamma_r(\R)$. 
 Since $g(x)=g(\gamma_0(0))\in \gamma_r(\R)$
 there uniquely exists $\tsl\in \R$ such that $g(x)=\gamma_r(\tsl)$. 
 By~\cref{lem:parallel-const},
 \begin{align*}
  \ds{g(\gamma_0(0)),\gamma_0(\tsl)} = \ds{\gamma_r(\tsl),\gamma_0(\tsl)} 
  = \ds{\gamma_r(0),\gamma_0(0)} = \Hds{\gamma_r,\gamma_0} = r.
 \end{align*}

 We show inductively that for $n\in \Z$, we have 
 \begin{align*}
  g^n(x)= g^{n}\circ \gamma_0(0) =  \gamma_{rn}(\tsl n).
 \end{align*}
 We suppose that $g^{n-1}\circ \gamma_0(0) =  \gamma_{r(n-1)}(\tsl (n-1))$ 
 for positive $n$. Then $g^{n-1}\circ \gamma_0(\tsl) =  \gamma_{r(n-1)}(\tsl n)$.
 So we have
 \begin{align*}
  \ds{g^n\circ \gamma_0(0),\gamma_{r(n-1)}(\tsl n)} &= 
  \ds{g^n\circ \gamma_0(0),g^{n-1}\circ \gamma_0(\tsl)} \\
  &= \ds{g\circ \gamma_0(0),\gamma_0(\tsl)} \\
  &= r \\
  &= \Hds{\gamma_{nr},\gamma_{(n-1)r}}.
 \end{align*}
 Then we have $g^n\circ \gamma_0(0) = \gamma_{rn}(\tsl n)$. 
 For negative $n$, we can prove in the same way.

 Now by \cref{lem:flat-plane}, a map $[u\mapsto \beta(\tsl u,ru)=\gamma_{ru}(\tsl u)]$
 is a linearly reparametrised geodesic on which the orbit of $x$ by $g$ lies.
 It follows that this is an axis of $g$, and so $g$ is hyperbolic.
\end{proof}

\cref{prop:elipA,prop:elipB,prop:hypA,prop:hypB} implies \cref{thm:classfyIsometry}.



\bibliographystyle{amsplain} \bibliography{/Users/tomo/Library/tex/math}

\bigskip
\address{ Tomohiro Fukaya \endgraf
Department of Mathematical Sciences,
Tokyo Metropolitan University,
Minami-osawa Hachioji, Tokyo, 192-0397, Japan
}

\textit{E-mail address}: \texttt{tmhr@tmu.ac.jp}

\end{document}